\documentclass{amsart}
\usepackage{latexsym,amsfonts,amsmath,amstext,amssymb,amscd,graphicx}
\usepackage{lineno}

\newtheorem{theorem}{Theorem}[section]

\newtheorem{lmm}[theorem]{Lemma}

\newtheorem{prp}[theorem]{Proposition}
\newtheorem{thm}[theorem]{Theorem}

\theoremstyle{definition}

\newtheorem{rmr}[theorem]{Remark}
\newtheorem{dfn}[theorem]{Definition}
\newcommand{\inv}{^{-1}}
\newcommand{\di}{\displaystyle}
\newcommand{\prim}{^\prime}

\newcommand{\Lgrad}{L = \bigoplus_{g\in G}L_g}
\newcommand{\Agrad}{A = \bigoplus_{g\in G}A_g}
\newcommand{\grad}{=\bigoplus_{g\in G}}
\newcommand{\ZN}{\mathbb{Z}}
\newcommand{\SSn}{S(m;\underline{n})}
\newcommand{\Sn}{S(m;\underline{n})^{(1)}}
\newcommand{\Xn}{X(m;\underline{n})^{(\infty)}}
\newcommand{\Wn}{W(m;\underline{n})}
\newcommand{\HHn}{H(m;\underline{n})}
\newcommand{\Hn}{H(m;\underline{n})^{(2)}}
\newcommand{\KKn}{K(m;\underline{n})}
\newcommand{\Kn}{K(m;\underline{n})^{(1)}}

\newcommand{\Omn}{O(m;\underline{n})}
\newcommand{\Gmn}{\GL(m;\underline{n})}
\newcommand{\Spmn}{\Sp(m;\underline{n})}
\newcommand{\unn}{\underline{n}}
\newcommand{\ut}{\underline{t}}

\newcommand{\un}{\underline{1}}
\newcommand{\xa}{x^{(a)}}
\newcommand{\xb}{x^{(b)}}

\newcommand{\taun}{\tau(\underline{n})}
\newcommand{\veps}{\varepsilon}

\newcommand{\M}{\mathcal{M}(m;\underline{n})}

\newcommand{\Amn}{{A}(m;\underline{n})}\newcommand{\Hf}{\mathcal{H}(m;\underline{n})}
\newcommand{\Xf}{\mathcal{X}(m;\underline{n})}
\newcommand{\Wf}{\mathcal{W}(m;\underline{n})}
\newcommand{\Sf}{\mathcal{S}(m;\underline{n})}
\newcommand{\Kf}{\mathcal{K}(m;\underline{n})}
\newcommand{\Fmn}{\mathcal{F}(m;\underline{n})}
\newcommand{\Tf}{\mathcal{T}}
\newcommand{\wk}{\omega_K}

\newcommand{\wh}{\omega_H}
\newcommand{\ws}{\omega_S}

\DeclareMathOperator{\Id}{\operatorname{Id}}
\DeclareMathOperator{\Der}{\operatorname{Der}}

\DeclareMathOperator{\Sp}{\operatorname{Sp}}

\DeclareMathOperator{\GL}{\operatorname{GL}}
\DeclareMathOperator{\Span}{\operatorname{Span}}

\DeclareMathOperator{\Supp}{\operatorname{Supp}}
\DeclareMathOperator{\Aut}{\operatorname{Aut}}
\DeclareMathOperator{\ad}{\operatorname{ad}}

\begin{document}
\title{Gradings by Groups on Graded Cartan Type\\Lie Algebras}
\author{Jason McGraw}
\begin{abstract}In this paper we describe all gradings by abelian groups without elements of order $p$, where $p>2$ is the characteristic of the base field, on the simple graded Cartan type Lie algebras. 
\end{abstract}
\maketitle 
\section{Introduction}\label{sI}

Let $A$ be an algebra, $G$ a group and let $\Aut A$, $\Aut G$ be the automorphism groups of $A$ and $G$, respectively.

\begin{dfn}\label{d4} A \emph{grading by a group $G$} on an algebra $A$, also called a {\em $G$-grading}, is a decomposition  $A = \bigoplus_{g\in G} A_g$ where each $A_g$ is a subspace such that $A_{g'}A_{g''}\subset A_{g'g''}$
for all $g',\, g''\in G$. For each $g\in G$, we call the subspace $A_g$ the {\em homogeneous space} of degree $g$.  A subspace $V$ of $A$ is called {\em graded} if $V=\bigoplus_{g\in G}(V\cap A_g)$.
The set $\Supp A=\{ g\in G\, |\, A_g \neq 0\}$ is called the {\em support} of the grading.
\end{dfn}

For a grading by a group $G$ on a \emph{simple} Lie algebra $L$, it is well known that the subgroup generated by the support is abelian \cite[Lemma 2.1]{typea}. For the remainder of the paper we always assume without loss of generality that the group is generated by the support. If $A$ is finite-dimensional, this assumption implies that $G$ is finitely generated.

\begin{dfn}\label{d5}
Two gradings $\di\Agrad$ and
$A=\di\bigoplus_{h\in G}A_h\prim$ of an algebra $A$ are called {\em equivalent} if there exist $\Psi\in\Aut(A)$ and $\theta\in\Aut(G)$ such that $\Psi(A_g)=A_{\theta(g)}\prim$
for all $g\in G$.  If $\theta$ is the identity, we call the gradings {\em isomorphic}.\end{dfn}

\begin{dfn}\label{l1}Let $A = \bigoplus_{g\in G}A_g$ be a grading by a group $G$ on an algebra $A$
and $\varphi$ a group homomorphism of $G$ onto $H$. The {\em coarsening of the $G$-grading induced by $\varphi$} is the $H$-grading defined by $A = \bigoplus_{h\in H}\overline{A}_{h}$
where $$\overline{A}_{h}=\di\bigoplus_{g\in G,\ \varphi(g)=h}A_{g}.$$\end{dfn}

The task of finding all gradings on simple Lie algebras by abelian groups in the case of algebraically closed fields of characteristic zero is almost complete ---  see \cite{Eld} and also \cite{typebcd,gtwo,typea,typeg,typef,typed,lie2}. In the case of positive characteristic, a description of gradings on the classical simple Lie algebras, with certain exceptions, has been obtained in \cite{BK}, \cite{posb}. In the case of simple Cartan type Lie algebras, the gradings by $\ZN$ have been described in \cite{slafpc}. All of them, up to isomorphism, fall into the category of what we call \emph{standard} gradings (which are coarsenings of the canonical $\ZN^k$-gradings) --- see Definitions \ref{d8} and \ref{d9}. This paper will deal with gradings on the simple graded Cartan type Lie algebras by arbitrary abelian groups without $p$-torsion in the case where the base field $F$ (which is always assumed to be algebraically closed) has characteristic $p>2$. We restrict ourselves to the \emph{graded} Cartan type Lie algebras (i.e., those that have canonical $\ZN$-gradings).

We use the notation of \cite{slafpc}, which is our standard reference for the background on Cartan type Lie algebras.

Our main result is the following.

\begin{thm}\label{tM}
Let $L$ be a simple graded Cartan type Lie algebra over an algebraically closed field of characteristic $p>2$. If $p=3$, assume that $L$ is not isomorphic to $W(1;\un)$ or $H(2;(1,n_2))^{(2)}$. Suppose $L$ is graded by a group $G$ without elements of order $p$. Then the grading is isomorphic to a standard $G$-grading.
\end{thm}

The correspondence between the gradings on an algebra by finite abelian groups of order coprime to $p$ and finite abelian subgroups of automorphisms of this algebra is well known. Using the theory of algebraic groups, this extends to infinite abelian groups. Namely, a grading on an algebra $\Lgrad$ by a finitely generated abelian group without elements of order $p$ gives rise to an embedding of the dual group $\widehat{G}$ into $\Aut L$ using the following action:
$$\chi*y=\chi(g)y,\quad\mbox{for all }y\in L_g,\quad g\in G,\quad\chi\in\widehat{G}.$$
We will denote this embedding by $\eta:\widehat{G}\to\Aut L$, so
$$\eta(\chi)(y)=\chi*y.$$ 
If $L$ is finite-dimensional, then $\Aut L$ is an algebraic group, and the image $\eta(\widehat{G})$ belongs to the class of algebraic groups called quasi-tori. Recall that a \emph{quasi-torus} is an algebraic group that is abelian and that consists of semisimple elements. Conversely, given a quasi-torus $Q$ in $\Aut L$, we obtain the eigenspace decomposition of $L$ with respect to $Q$, which is a grading by the group of characters of $Q$, $G=\mathfrak{X}(Q)$.

In this paper, $L$ is a simple graded Cartan type Lie algebra, i.e., one of the following algebras: $\Wn$, $\Sn$, $\Hn$, $\Kn$ where $m$ is a positive integer and $n=(n_1,\dots,n_m)$ is an $m$-tuple of positive integers --- see the definitions in the next section. We will denote the type of $L$ by $\Xn$. The automorphism group of $L$ can be regarded as a subgroup of $\Aut_c\Omn$, the group of continuous automorphisms of the commutative divided power algebra $\Omn$. Since $\eta(\widehat{G})$ is a quasi-torus, a result by Platonov \cite{plat} tells us that $\eta(\widehat{G})$ is contained in the normalizer of a maximal torus. Starting from this result, we show that, in fact, $\eta(\widehat{G})$ is conjugate to a subgroup of the standard maximal torus $T_{X}$ (specific for each type $\Xn$ of simple graded Cartan type Lie algebra), which is responsible for the standard $\ZN^k$-grading on $L$ where $k$ depends on the type $X$. 

Unless it is stated otherwise, we denote by $a$ and $b$ some $m$-tuples of non-negative integers and by $i,j,k,l,q,r$ some integers. Any subset denoted by a calligraphic letter, for example $\Wf$, is a subset of $\Aut_c\Omn$.

\section{Cartan Type Lie Algebras and Their Standard Gradings}

In this section we introduce some basic definitions, closely following \cite[Chapter 2]{slafpc}. We start by defining the graded Cartan type Lie algebras $\Wn$, $\Sn$, $\Hn$, $\Kn$.

\begin{dfn}\label{d1}
Let $\Omn$ be the commutative algebra $$\Omn:=\left\{\di\sum_{0\leq a\leq\taun}\alpha(a)\xa\;|\;\alpha(a)\in F\right\}$$ over a field $F$ of characteristic $p$, where $\taun=(p^{n_1}-1,\dots,p^{n_m}-1)$, with multiplication $$\xa\xb=\binom{a+b}{a}x^{(a+b)},$$
where $\di\binom{a+b}{a}=\di\prod_{i=1}^m\binom{a_i+b_i}{a_i}$.

For $1\leq i\leq m$, let $\epsilon_i:=(0,\dots,0,1,0\dots, 0)$, where the 1 is at the $i$-th position, and let $x_i:=x^{(\epsilon_i)}$.
\end{dfn}

There are standard derivations on $\Omn$ defined by $\partial_i(\xa)=x^{(a-\veps_i)}$ for $1\leq i\leq m$. 

\begin{dfn}\label{d2} Let $\Wn$ be the Lie algebra $$\Wn:=\left\{\di\sum_{1\leq i\leq m}f_i\partial_i\;|\;f_i\in\Omn\right\}$$ with the commutator defined by $$[f\partial_i,g\partial_j]=f(\partial_i g)\partial_j-g(\partial_jf)\partial_i,\quad f,g\in\Omn.$$
\end{dfn} 

The Lie algebras $\Wn$ are called \emph{Witt algebras}. $\Wn$ is a subalgebra of $\Der\Omn$, the Lie algebra of derivations of $\Omn$.

The remaining graded Cartan type Lie algebras are subalgebras of $\Wn$. When dealing with Hamiltonian and contact algebras  in $m$ variables (types $H(m;\underline{n})$ and $K(m;\underline{n})$ below), it is useful to introduce the following notation: 
$$i\prim=\left\{\begin{array}{ll}i+r,&\mbox{if }1\leq i\leq r\\
i-r,&\mbox{if }r+1\leq i\leq 2r,\\
\end{array}\right.$$
$$\sigma(i)=\left\{\begin{array}{rl}1,&\mbox{if }1\leq i\leq r\\
-1,&\mbox{if }r+1\leq i\leq 2r,\\
\end{array}\right.$$
where $m=2r$ in the case of $H(m;\unn)$ and $2r+1$ in the case of $K(m;\unn)$. Note that we do not define $m\prim$ or $\sigma(m)$ if $m=2r+1$. We will also need the following differential forms --- see \cite[Section 4.2]{slafpc}.

$$\begin{array}{llll}
\ws&:=&dx_1\wedge\dots\wedge dx_m,\quad\quad\quad\quad\quad& m\geq3,\\
\wh&:=&\di\sum_{i=1}^rdx_i\wedge dx_{i\prim},&m=2r,\\
\wk&:=&dx_m+\di\sum_{i=1}^{2r}\sigma(i)x_idx_{i\prim},&m=2r+1.
\end{array}$$

\begin{dfn}\label{d3}
We define the {\em special, Hamiltonian} and {\em contact algebras} as follows:
$$\begin{array}{lllll}
\SSn&:=&\{D\in \Wn\;|\;D(\ws)=0\},&m\geq3,\\
\HHn&:=&\{D\in \Wn\;|\;D(\wh)=0\},&m=2r,\\
\KKn&:=&\{D\in \Wn\;|\;D(\wk)\in\Omn\wk\},&m=2r+1,\\
\end{array}$$
respectively.\end{dfn}

The algebras of Definitions \ref{d2} and \ref{d3}, as well as their derived subalgebras, are collectively referred to as \emph{graded Cartan type Lie algebras}.

It is known that the Lie algebras $\Wn$ are simple, but $\SSn$ and $\HHn$ are not simple, and $\KKn$ are simple if and only if $p$ divides $m+3$. The first derived algebras $\Sn$ and $\Kn$, and second derived algebras $\Hn$ are simple. The Cartan type Lie algebras defined above are called graded, because they have a canonical $\ZN$-grading, defined by declaring $$\deg(x^{(a)}\partial_i)=a_1+\cdots+a_m-1$$
 for types $\Wn$, $S(m;\unn)$, $H(m;\unn)$, and $$\deg(x^{(a)}\partial_i)=a_1+\cdots+a_{m-1}+2a_m-1-\delta_{i,m}$$
for $K(m;\unn)$. The $\ZN$-grading on $\Wn$ is a coarsening of the following $\ZN^m$-grading.

\begin{dfn}\label{d6}
The $\ZN^m$-gradings on $\Omn$ and $\Wn$,	

$$\Omn=\bigoplus_{a\in\ZN^m}\Omn_{a},$$	
$$\Wn=\bigoplus_{a\in\ZN^m}\Wn_{a},$$	
where $$\Omn_{a}=\Span\{x^{(a)}\},$$	
$$\Wn_a=\Span\{x^{(a+\epsilon_k)}\partial_k\;|\;1\leq k\leq m\},$$
are called the {\em canonical $\ZN^m$-gradings on} $\Omn$ and $\Wn$, respectively.
\end{dfn}\label{d7}

Note that in the above grading on $\Wn$ the support includes tuples with negative entries. For example $\Wn_{-\veps_i}=\Span\{\partial_i\}$. The algebras $S(m;\underline{n})$ and $\Sn$ are graded subspaces in the canonical $\ZN^m$-grading on $\Wn$, so they inherit the {\em canonical $\ZN^m$-grading}.

\begin{dfn}\label{d8}
Let $G$ be an abelian group, $L=\Wn$ or $\Sn,$ and $\varphi$ a homomorphism $\ZN^m\to G$. The decompositions $\Omn=\bigoplus_{g\in G} O_g$, $L=\bigoplus_{g\in G} L_g$, 
given by
$$O_g=\Span\{\xa\;|\;\varphi(a)=g\},$$
$$L_g=\Span\{\xa\partial_k\;|\;1\leq k\leq m,\, \varphi(a-\epsilon_k)=g\}\cap L,$$ are $G$-gradings on $\Omn$ and $L$, respectively. We call them the \emph{standard $G$-gradings induced by $\varphi$} on $\Omn$ and $L$, respectively. 
We will refer to the standard $G$-grading induced by any $\varphi$ as a \emph{standard $G$-grading} when $\varphi$ is not specified.  
\end{dfn}

Let $L=\Wn$ and let $\Lgrad$ be the standard $\ZN^m$-grading induced by $\varphi$. Let $\varphi(\veps_i)=g_i\in G$. The corresponding action of $\widehat G$ on $L$ is defined by $$\chi*(\xa\partial_i)=\chi(\varphi(a-\veps_i))\xa\partial_i=\chi(g_1)^{a_1}\cdots\chi(g_m)^{a_m}\chi(g_i)^{-1}\xa\partial_i,$$
for all $\chi\in\widehat{G}$. Hence $\eta(\widehat G)$ is a subgroup of the torus $T_W$,
$$T_W:=\{\Psi\in\Aut\Wn\;|\;\Psi(\xa\partial_k)=t_1^{a_1}\cdots t_m^{a_m}t_k^{-1}\xa\partial_k,\ t_j\in F^\times\}.$$
Conversely, if $Q$ is a quasi-torus in $T_W$, it defines a standard grading on $L$ by $G=\mathfrak{X}(Q)$, the group of characters of $Q$.

In particular, the standard $\ZN^m$-grading on $\Wn$ corresponds to $Q=T_W$. Note that $T_W$ preserves the subalgebra $\Sn$, and the restriction of $T_W$ to $\Sn$ is an isomorphic torus in $\Aut\Sn$.

\begin{lmm}\label{lt}\cite[Section 7.4]{slafpc}
The following are maximal tori of $\Aut\Wn$, $\Aut\Sn$, $\Aut\Hn$ and $\Aut\Kn$, respectively: 
$$\begin{array}{lll}T_W&=&T_S=\{\Psi\in\Aut\Wn\;|\;\Psi(\xa\partial_i)=t_1^{a_1}\cdots t_m^{a_m}t_i\inv\xa\partial_i,\ t_i\in F^\times\},\\
\\ 
T_H&=&\{\Psi\in\Aut\Wn\;|\;\Psi(\xa\partial_i)=t_1^{a_1}\cdots t_m^{a_m}t_i\inv\xa\partial_i,\ t_i\in F^\times,\\ 
&&\ \ t_it_{i\prim}=t_jt_{j\prim},\ 1\leq i,j\leq r\},\\
\\
T_K&=&\{\Psi\in\Aut\Wn\;|\;\Psi(\xa\partial_i)=t_1^{a_1}\cdots t_m^{a_m}t_i\inv\xa\partial_i,\ t_i\in F^\times,\\ &&\ \  t_it_{i\prim}=t_jt_{j\prim}=t_{m},\ 1\leq i,j\leq r\}.\end{array}$$$\hfill\square$
\end{lmm}

We are now ready to define standard $G$-gradings on $\Hn$ and $\Kn$. They are obtained by coarsening the canonical $\ZN^m$-grading on $\Wn$. The homomorphism $\varphi:\ZN^m\to G$ must satisfy certain conditions in order for $\eta(\widehat G)$ to preserve $\Hn$ and $\Kn$, respectively.

The {\em canonical $\ZN^{r+1}$-grading on $H(2r;\underline{n})$} is the restriction of a coarsening of the $\ZN^{m}$-grading on $W(2r;\unn)$, defined as follows. Let $m=2r$ and let $$\phi_H:\ZN^m\to\ZN^m/\langle\veps_i+\veps_{i\prim}=\veps_j+\veps_{j\prim}\, |\, 1\leq i<j\leq r\rangle$$
be the quotient map. Then the coarsening of the canonical $\ZN^m$-grading on $\Wn$ induced by $\phi_H$ is a $\ZN^{r+1}$-grading, which restricts to the subalgebras $H(m;\underline{n})$, $H(m;\underline{n})^{(1)}$ and $\Hn$.

Similarly, the {\em canonical $\ZN^{r+1}$-grading on $K(2r+1;\underline{n})$} is the restriction of a coarsening of the $\ZN^{m}$-grading on $W(2r+1;\unn)$, defined as follows. Let $m=2r+1$ and $$\phi_K:\ZN^m\to\ZN^m/\langle\veps_i+\veps_{i\prim}=\veps_m\, |\, 1\leq i\leq r\rangle$$
be the quotient map. Then the coarsening of the canonical $\ZN^m$-grading on $\Wn$ induced by $\phi_K$ is a $\ZN^{r+1}$-grading, which restricts to the subalgebras $K(m;\underline{n})$ and $\Kn$. 

\begin{dfn}\label{d9}
Let $L=\Hn$ or $\Kn$ and $X=H$ or $K$, respectively. Let
$G$ be an abelian group, and $\theta$ a group homomorphism from $\varphi_X(\ZN^m)$ to $G$. The decomposition $L=\bigoplus_{g\in G} L_g$, 
given by
$$L_g=\Span\{\xa\partial_k\;|\;1\leq k\leq m,\, \theta\varphi_X(a-\epsilon_k)=g\}\cap L,$$ is a $G$-grading on $L$. We call it the \emph{standard $G$-grading on $L$ induced by $\theta$}, or just a standard $G$-grading on $L$ if $\theta$ is not specified.
\end{dfn}

We can summarize the above discussion as follows.

\begin{lmm}\label{le} Let $L=\Wn$, $\Sn$, $\Hn$ or $\Kn$. A grading by a group $G$ on $L$ is a standard $G$-grading if and only if we have $\eta(\widehat{G})\subset T_X$ where $X=W$, $S$, $H$ or $K$, respectively.$\hfill\square$\end{lmm}

The goal of Section 4 is to show that if $G$ has no elements of order $p$ then $\eta(\widehat{G})$ is always containted in a maximal torus. Hence, a conjugate of $\eta(\widehat{G})$ will be contained in $T_X$, which will mean, in view of Lemma \ref{le}, that the $G$-grading is isomorphic to a standard $G$-grading.

\section{The automorphism groups of simple graded Cartan type\\ Lie algebras}\label{sA}

The automorphism group of each simple graded Cartan type Lie algebra is a subgroup of the automorphism group of the Witt algebra, which in turn is isomorphic to a subgroup of the automorphism group of $O(m;\underline{n})$. We will describe the automorphism groups of the Cartan type Lie algebras as subgroups of the automorphism groups of $O(m;\underline{n})$.

We start by introducing so-called {\em continuous} automorphisms of $\Omn$.
\begin{dfn}\label{d11}
Let $\Amn$ be the set of all $m$-tuples $(y_1,\dots,y_m)\in\Omn^m$ for which $\det(\partial_i(y_j))_{1\leq i,j\leq m}$ is invertible in $\Omn$ and also
$$y_i=\di\sum_{0<a\leq\taun}\alpha_i(a)\xa\quad\mbox{with }\alpha_i(p^l\epsilon_j)=0\mbox{ if }n_i+l>n_j.$$
\end{dfn}
The group of continuous automorphisms of $\Omn$, $\Aut_c\Omn$, is defined as follows \cite[Theorem 6.32]{slafpc}. For any $(y_1,\dots,y_m)\in\Amn$ we define a map $\Omn\to\Omn$ by setting
$$\varphi\left(\di\sum_{0\leq a \leq\taun}\alpha(a)x^{(a)}\right)=\di\sum_{0\leq a\leq\taun}\alpha(a)\di\prod_{i=1}^my_i^{(a_i)},$$
where $y^{(q)}$, $q\in\mathbb{N}$, denotes the $q^{\mathrm{th}}$ divided power of $y$ \cite[Chapter 2]{slafpc}. 

We have a map $\Phi$ from $\Aut_c\Omn$ to $\Aut \Wn$ defined by $$\Phi(\psi)\left(\di\sum_{1\leq i\leq m}f_i\partial_i\right)=\psi\circ\left(\di\sum_{1\leq i\leq m}f_i\partial_i\right)\circ\psi^{-1},$$ where the elements of $\Wn$ are viewed as  derivations of $\Omn$.

\begin{thm}\label{t2}\emph{\cite[Theorem 7.3.2]{slafpc}}
The map $\Phi:\Aut_c\Omn\to\Wn$ is an isomorphism of groups provided that $(m;\underline{n})\neq (1,1)$ if $p=3$. Also, except for the case of $H(m;n)^{(2)}$ with $m=2$ and $\min(n_1,n_2)=1$ if $p=3$,
$$\begin{array}{lllll}\Aut \Sn&=&\Phi(\{\psi\in\Aut_c\Omn\;|\;\psi(\ws)\in F^\times\ws\}),\\
\Aut \Hn&=&\Phi(\{\psi\in\Aut_c\Omn\;|\;\psi(\wh)\in F^\times\wh\}),\\
\Aut \Kn&=&\Phi(\{\psi\in\Aut_c\Omn\;|\;\psi(\wk)\in \Omn^\times\wk\})\end{array}$$
$\hfill\square$
\end{thm}

\noindent
\begin{rmr}The map of the tangent Lie algebras corresponding to $\Phi$ is a restriction of $\ad:W\to\Der W$, and hence injective. It follows that $\Phi$ is an isomorphism of {\em algebraic} groups.\end{rmr}

We will use the following notation 
$$\begin{array}{lllll} \Sf&=&\{\psi\in\Aut_c\Omn\;|\;\psi(\ws)\in F^\times\ws\},\\
\Hf&=&\{\psi\in\Aut_c\Omn\;|\;\psi(\wh)\in F^\times\wh\},\\
\Kf&=&\{\psi\in\Aut_c\Omn\;|\;\psi(\wk)\in \Omn^\times\wk\}.\end{array}$$
The above groups are subgroups of $\Wf:=\Aut_c\Omn$. We will refer to them collectively by $\Xf$ where $\mathcal{X}=\mathcal{W},$ $\mathcal{S},$ $\mathcal{H}$ or $\mathcal{K}$. 

\section{Gradings by groups without $p$-torsion}\label{sC}

As mentioned in Section 2, to prove our Theorem \ref{tM}, we need to show that $\eta(\widehat{G})$ is conjugate in $\Aut\Xn$ to a subgroup of the maximal torus $T_X$ where $X=W$, $S$, $H$, $K$. We are going to use an important general result following from \cite[Corollary 3.28]{plat}.

\begin{prp}\label{t10}
A quasi-torus of an algebraic group is contained in the normalizer of a maximal torus. $\hfill\square $
\end{prp}

This brings us to the necessity of looking at normalizers of maximal tori in $\Aut\Xn$. Using the isomorphism $\Phi$ described in Section 3, we are going to work inside the groups $\Wf=\Aut_c\Omn$.

\subsection{Normalizers of maximal tori}\label{ssN}

We denote by $\Aut_0\Omn$ the subgroup of $\Aut_c\Omn$ consisting of all $\psi$ such that $\psi(x_i)=\di\sum_{j=1}^m\alpha_{i,j}x_{j}$, $\alpha_{i,j}\in F$, $1\leq j\leq m$. The group $\Aut_0\Omn$ is canonically isomorphic to a subgroup of $\GL(m)$, which we denote by $\Gmn$. If $n_i=n_j$ for $1\leq i,j\leq m$ then $\Gmn=\GL(m)$, otherwise $\Gmn\neq \GL(m)$. The condition for a tuple $(y_1,
\dots,y_n)$ to be in $\Amn$, \begin{equation}\label{flag}y_i=\di\sum_{0<a}\alpha_i(a)\xa\quad\mbox{with }\alpha_i(p^l\epsilon_j)=0\mbox{ if }n_i+l>n_j,\end{equation}
imposes a {\em flag} structure on the vector space $V=\Span\{x_1,\dots,x_m\}$.

\begin{dfn}\label{d21}
Given $\underline{n}=(n_1,\dots,n_m)$, with $m>0$, we set  $\Xi_0=\emptyset$ and then, inductively,
$$\Xi_i=\Xi_{i-1}\cup\{j\;|\;n_j=\di\max_{k\notin \Xi_{i-1}}\{n_k\}\}.$$
Set $V_i=\Span\{x_j\;|\;j\in \Xi_i\}\mbox{ for }i\geq 0$. Then $0=V_0\subset V_1\subset V_2\subset\cdots\subset V$ is a \emph{flag} in $V$ (i.e., an ascending chain of subspaces). We denote this flag by $\Fmn$ and say that an automorphism $\psi$ of $\Omn$ {\em respects} $\Fmn$ if $\psi(V_i)=V_i$, for all $i$.
\end{dfn}

\noindent Condition (\ref{flag}) implies that $\Gmn$ consists of all elements of $\GL(m)$ that respect $\Fmn$.

According to \cite[Section 7.3]{slafpc}, $\Aut_0\Omn\cap\Sf=\Aut_0\Omn$, i.e., in the case of special algebras we have to deal with the same subgroup of $\GL(m)$ as in the case of Witt algebras.

In the Hamiltonian case, $V = \Span \{x_1,\ldots,x_r\} \oplus \Span \{ x_{1'}, \ldots , x_{r'}\}$, and $\wh$ induces a nondegenerate skew-symmetric form on $V$, given by $\langle x_i,x_j\rangle=\sigma(i)\delta_{i,j\prim}$, for all $i,j=1,\ldots,2r$. The image of $\Aut_0\Omn\cap\Hf$ in $\GL(m)$, $m=2r$, is the product of the subgroup of scalar matrices and the subgroup $\Spmn:=\Sp(m)\cap\Gmn$. This product is almost direct: the intersection is $\{\pm\Id\}$.

The maximal tori $T_X$ in $\Aut\Xn$ described in Lemma \ref{lt} correspond, under the algebraic group isomorphism $\Phi$, to the following maximal tori in $\Xf$:
$$\begin{array}{lcl}\Tf_W&=&\Tf_S=\{\psi\in\Wf\;|\;\psi(x_i)=t_ix_i,\ t_i\in F^\times\},\\\\
\Tf_H&=&\{\psi\in\Wf\;|\;\psi(x_i)=t_ix_i,\ t_i\in F^\times,\, t_it_{i\prim}=t_jt_{j\prim},\},\\\\
\Tf_K&=&\{\psi\in\Wf\;|\;\psi(x_i)=t_ix_i,\ t_i\in F^\times,\, t_it_{i\prim}=t_jt_{j\prim}=t_{m},\ 1\leq i,j\leq r\}.\end{array}$$

A convenient way to view the elements of the above tori is to view them as $m$-tuples of nonzero scalars. Define $\lambda:(F^\times)^m\to \Aut_0\Omn$ where $\lambda(\underline{t})(x_i)=t_ix_i$ for $1\leq i\leq m$. Then $\lambda((F^\times)^m)=\Tf_W$.

\begin{dfn}\label{d30}
We will say that $\underline{t}\in(F^\times)^m$ is {\em $X$-admissible} if $\lambda(\underline{t})\in \Tf_X$, where $X=W,$ $S$, $H$ or $K$.
\end{dfn}

An important subgroup which we use for the description of the normalizer $N_{\Xf}(\Tf_X)$ for $X=W,S,H$ or $K$, is the subgroup $\M$ of $\Wf$.

\begin{dfn}\label{d31}
Let $\M$ be the subgroup of $\Wf$ that consists of $\psi$ such that, for each $1\leq i\leq m$, we have $\psi(x_i)=\alpha_ix_{j_i}$ where $\alpha_i\in F^\times$ and $1\leq j_i\leq m$.\end{dfn}

Thus $\M\subset\Aut_0\Omn$ is isomorphic to the group of monomial matrices that respect the flag.

\begin{lmm}\label{l22}
The subgroups $N_{\Wf}(\Tf_W)$, $N_{\Sf}(\Tf_S)$ and $N_{\Hf}(\Tf_H)$ are contained in $\M$.
\end{lmm}

\begin{proof}

We will show that $N_{\Wf}(\Tf_X)\subset\M$ for $X=W,$ $S,$ $H$. Since $\Xf\subset\Wf$ we have $N_{\Xf}(\Tf_X)\subset N_{\Wf}(\Tf_X)$.

Let $\psi\in N_{\Wf}(\Tf_X)$. For any $1\leq i\leq m$ the element $x_i$ is a common eigenvector of $\Tf_X$ so $\psi(x_i)$ is also a common eigenvector of $\Tf_X$. Also, since $\psi\in\Aut_c \Omn$, $\psi(x_i)=\di\sum_{0<a\leq\tau(\underline{n})}\alpha_i(a)x^{(a)}$ where, among other conditions,
$\alpha_i(\epsilon_{j_i})\neq 0$ for some $1\leq j_i\leq{m}$.

First we consider the case $X=W$ and $X=S$ (recall that $\Tf_W=\Tf_S$).

It is easy to see that the eigenspace decomposition of $\Omn$ with respect to $\Tf_W$ is the canonical $\ZN^m$-grading on $\Omn$. The homogeneous space $O_a=\Span\{\xa\}$ is the eigenspace with eigenvalue $\underline{t}^{a}:=t_1^{a_1}\cdots t_m^{a_m}$ with respect to $\lambda(\underline{t})\in \Tf_W$. It follows that $\psi(x_i)\in O_a$ for some $0\leq a\leq\taun$ since $\psi(x_i)$ is an eigenvector of $\Tf_X$. The condition that $\alpha_i(\epsilon_{j_i})\neq 0$ for some $1\leq j_i\leq m$ forces $a=\veps_{j_i}$. Hence $\psi\in\M$.

We continue with the case of $X=H$.

The torus $\Tf_H$ is contained in $\Tf_W$. In order for $\lambda(\underline{t})\in \Tf_W$ to belong to $\Tf_H$, the $m$-tuple $\ut$ must be $H$-admissible, i.e., satisfy $t_it_{i\prim}=t_jt_{j\prim}$ for $1\leq i,j\leq r$ where $m=2r$. The eigenspace decomposition of $\Omn$ with respect to $\Tf_H$ is a coarsening of the canonical $\ZN^m$-grading. The eigenspace with eigenvalue $\underline{t}^{a}$ with respect to $\lambda(\underline{t})\in\Tf_H$ is $Q_a:=\bigoplus O_b$ where the direct sum is over the set of all $b$ such that $\underline{t}^{a}=\underline{t}^{b}$ for all $H$-admissible $\ut$. If $O_b\neq0$ and $t_k=\underline{t}^{b}$ for all $H$-admissible $\underline{t}$, then $b=\veps_k$ since all the entries in the $m$-tuple $b$ are non-negative. This implies that $Q_{\veps_k}=\Span\{x_k\}$. Now $\psi(x_i)\in Q_a$ for some $0\leq a\leq\taun$ since $\psi(x_i)$ is an eigenvector of $\Tf_H$. The condition that $\alpha_i(\epsilon_{j_i})\neq 0$ for some $1\leq j_i\leq 2r$ again forces $a=\veps_{j_i}$. Hence $\psi\in\M$. \end{proof}
For the case of contact algebras, similar arguments do not give us that\\ $N_{\Wf}(\Tf_K)$ is in $\M$.

\begin{lmm}\label{l23}
If $\psi\in N_{\Kf}(\Tf_K)$, $m=2r+1$ then  \begin{equation}\label{e2}\psi(x_m)=\alpha_m(\veps_m)x_m+\di\sum_{l=1}^r\alpha_m(\veps_l+\veps_{l\prim})x_lx_{l\prim}\end{equation}and, for $1\leq i\leq 2r$, we have $\psi(x_i)=\alpha_i(\veps_{j_i})x_{j_i}$ where $1\leq j_i\leq 2r$.
\end{lmm}

\begin{proof} 

We will prove that any $\psi\in N_{\Wf}(\Tf_K)$ has the form above. Since $x_i$ is a common eigenvector of $\Tf_K$ we have that $\psi(x_i)$ is also a common eigenvector of $\Tf_K$. Since $\psi\in\Aut_c\Omn$, we have $\psi(x_i)=\di\sum_{0<a\leq\tau(\underline{n})}\alpha_i(a)x^{(a)}$ where, among other conditions,
$\alpha_i(\epsilon_{j_i})\neq 0$ for some $1\leq j_i\leq{m}$.

The torus $\Tf_K$ is contained in $\Tf_W$. In order for a $\lambda(\underline{t})$ to be in $\Tf_K$, the $m$-tuple $\ut$ must be $K$-admissible, i.e., $t_it_{i\prim}=t_m$ for $1\leq i\leq r$. The eigenspace decomposition of $\Omn$ with respect to $\Tf_K$ is a coarsening of the canonical $\ZN^m$-grading. The eigenspace with eigenvalue $\underline{t}^{a}$ with respect to $\lambda(\underline{t})\in\Tf_K$ is \\
$R_a=\bigoplus O_b$ where the direct sum is over the set of $b$ such that $\underline{t}^{a}=\underline{t}^{b}$ for all $K$-admissible $\underline{t}$. If $O_b\neq0$ and $\underline{t}^{b}=t_k$ for all $K$-admissible $\underline{t}$,  $1\leq k\leq 2r$, then $b=\veps_k$ since all entries of the $m$-tuple $b$ are non-negative. This implies that $R_{\veps_k}=\Span\{x_k\}$ for $1\leq k\leq 2r$. If $O_b\neq0$ and $\underline{t}^{b}=t_m$ for all $K$-admissible $\underline{t}$ then either $b=\veps_m$ or $b=\veps_i+\veps_{i\prim}$ where $1\leq i\leq r$. This implies that $R_{\veps_m}=\Span\{x_m,x_ix_{i\prim}\ |\ 1\leq i\leq r\}$.

Now $\psi(x_i)\in R_a$ for some $0\leq a\leq\taun$. The condition that $\alpha_i(\epsilon_{j_i})\neq 0$ for some $1\leq j_i\leq{m}$ forces $a=\veps_{j_i}$. Note that the dimension of $R_{\veps_i}$ is 1 for $1\leq i\leq 2r$ and the dimension of $R_{\veps_m}$ is $r+1$. Hence, for $1\leq i\leq 2r$, we have $\psi(x_i)=\alpha_i(\veps_{j_i})x_{j_i}$ for some $1\leq j_i\leq 2r$. Also, $\psi(x_m)\in R_{\veps_m}$ which implies (\ref{e2}).\end{proof}

\subsection{Diagonalization of quasi-tori}\label{ssIN}

In this subsection we prove that any quasi-torus in $\Aut\Xn$ is conjugate to a subgroup of the maximal torus $T_X$, where $X$ is $W$, $S$, $H$ or $K$. As before we pass from $\Aut\Xn$ to $\Xf$ by virtue of the isomorphism $\Phi$ described in Section 3.

\begin{prp}\label{p1}
Let $\mathcal{Q}$ be a quasi-torus contained in $N_{\Xf}(\Tf_X)$ where 
$X=W$ or $S$. Then there exists $\psi \in\Xf$ such that $\psi\mathcal{Q}\psi\inv\subset \Tf_X$. 
\end{prp}

\begin{proof}
Recall the flag $\Fmn$, $$V_0\subset V_1\subset\cdots\subset V.$$
Let $U_i=\Span\{x_{j}\;|\;x_j\in V_i,\,x_j\notin V_{i-1}\}$. Then $V_i=\di\bigoplus_{j=1}^iU_j$. Since $\mathcal{Q}\subset \M$ we have $\mathcal{Q}(U_i)=U_i$. (Here, as before, we identify $\Aut_0\Omn$ with a subgroup of $\GL(m)$.)

Restricting the action of $\mathcal{Q}$ to $U_i$, we obtain a subgroup of $\GL(U_i)$. Since $\mathcal{Q}|_{U_i}$ is a quasi-torus, there exists a $P_i\in\GL(U_i)$ such that $P_i(\mathcal{Q}|_{U_i})P_i\inv$ is diagonal. We can extend the action of $P_i$ to the whole space $V$ by setting $P_i(y)=y$ for all $y\in U_j$, $i\neq j$. These extended $P_i$ respect $\Fmn$. The product of these transformations, $P=P_1\cdots P_l$, is an element of $\Aut_0\Omn$, which diagonalizes $\mathcal{Q}$.
\end{proof}

In order to obtain the analog of Proposition \ref{p1} in the case of Hamiltonian algebras, we consider the canonical skew-symmetric inner product $\langle \,,\,\rangle$ on $V$ given by $\langle x_j,x_k\rangle=\sigma(j)\delta_{j,k\prim}$, for all $j,k=1,\ldots,2r$.

\begin{lmm}\label{l24}
Let $\mathcal{Q}$ be a quasi-torus contained in $\Spmn$. Then there is a basis $\{e_i\}_{i=1}^{2r}$ of $V$ such that $\langle e_j,e_k\rangle=\sigma(j)\delta_{j,k\prim}$, all $e_j$ are common eigenvectors of $\mathcal{Q}$, and $V_i=\Span\{e_j\;|\;j\in \Xi_i\}$ for all $i$.
\end{lmm}

\begin{proof}

We can decompose $V=\bigoplus V^{\gamma}$
where $V^{\gamma}$ are the eigenspaces in $V$ with respect to $\mathcal{Q}$, indexed by $\gamma\in\mathfrak{X}(\mathcal{Q})$ where $\mathfrak{X}(\mathcal{Q})$ is the group of characters of $\mathcal{Q}$. Since $\mathcal{Q}(V_i)=V_i$ for any $i$, there is a basis $\{y_j\}_{j=1}^{2r}$ such that $y_j$ are eigenvectors of $\mathcal{Q}$ and each $V_i=\Span\{y_j\;|\;j\in \Xi_i\}$.

We have $\langle x_j,x_k\rangle=\sigma(j)\delta_{j,k\prim}$. We will show by induction on $r$ that there is a basis $\{e_j\}_{j=1}^{2r}$ such that $\langle e_j,e_k\rangle=\sigma(j)\delta_{j,k\prim}$, all $e_j$ are common eigenvectors of $\mathcal{Q}$, and $V_i=\Span\{e_j\;|\;j\in \Xi_i\}$. The base case $r=1$ is obvious.

We have a basis $\{y_j\}_{j=1}^{2r}$ such that $y_j$ are eigenvectors of $\mathcal{Q}$ and each $V_i=\Span\{y_j\;|\;j\in \Xi_i\}$. We apply a process similar to the Gram--Schmidt process to find a new basis of common eigenvectors for $\mathcal{Q}$ that satisfies the desired conditions.

Since $V_1\neq0$ and $\langle,\rangle$ is nondegenerate, $\langle V_1,V_l\rangle\neq0$ for some $l$. Let $l$ be minimal, i.e., $\langle V_1,V_l\rangle\neq0$ and $\langle V_1,V_i\rangle=0$ for $i<l$. So there exist $y_s\in V_1$ and $y_t\in V_l$ such that $\langle y_s,y_t\rangle\neq 0$ and $\langle y_s, V_i\rangle=0$ if $i<l$ by the minimality of $l$. Let $\gamma_j\in\mathfrak{X}(\mathcal{Q})$ be the eigenvalue of $y_j$. Since $\mathcal{Q}$ consists of symplectic transformations, we have $\gamma_s=\gamma_t^{-1}$.

 For $j\neq s,t$, let 
$$z_j=\langle y_s,y_t\rangle y_j-\langle y_s,y_j\rangle y_t+\langle y_t,y_j\rangle y_s.$$

The $z_j$ with $y_s$ and $y_t$ form a basis of $V$. They also satisfy

$$\langle y_s,z_j\rangle=\langle y_s,y_t\rangle \langle y_s,y_j\rangle-\langle y_s,y_j\rangle \langle y_s,y_t\rangle+\langle y_t,y_j\rangle \langle y_s,y_s\rangle=0,$$
and similarly $\langle y_t,z_j\rangle=0$.

We also have the property that $z_j\in V_i$ if and only if $y_j\in V_i$. Indeed, for $i<l$ and $y_j\in V_i$, we have $\langle y_s,y_j\rangle=0$ by the minimality of $l$. This shows that $z_j$ is in $V_1+V_i$ and hence $z_j\in V_i$. For $i\geq l$ and $y_j\in V_i$ we have $y_j,y_t,y_s\in V_i$ which implies $z_j\in V_i$. Finally, we want to verify that $z_j$ are common eigenvectors of $\mathcal{Q}$. There are three cases to consider. 

\noindent Case 1: $\gamma_j\neq\gamma_s^{\pm1}$

Recall that $\gamma_s=\gamma_t\inv$. Since $\gamma_j\neq\gamma_s^{\pm1}$, we have $\langle y_s,y_j\rangle=\langle y_t,y_j\rangle=0$. This means $z_j=\langle y_s,y_t\rangle y_j$, which is an eigenvector with eigenvalue $\gamma_j$.

\noindent Case 2: $\gamma_j=\gamma_s$ or $\gamma_s\inv$, and  $\gamma_s\neq\gamma_s\inv$.

Suppose $\gamma_j=\gamma_s$. Since $\gamma_j\neq\gamma_s\inv$, we have $\langle y_s,y_j\rangle=0$. This means $z_j=\langle y_s,y_t\rangle y_j+\langle y_t,y_j\rangle y_s,$ which is an eigenvector with eigenvalue $\gamma_j=\gamma_s$. A similar argument applies if $\gamma_j=\gamma_s\inv$.

\noindent Case 3: $\gamma_j=\gamma_s=\gamma_s\inv$.

Since $z_j=\langle y_s,y_t\rangle y_j-\langle y_s,y_j\rangle y_t+\langle y_t,y_j\rangle y_s$ and $\gamma_s=\gamma_s\inv=\gamma_t$, we see that $z_j$ is an eigenvector with eigenvalue $\gamma_j$.

In order to use the induction hypothesis we relabel our basis as follows: Pick $x_q\in V_1$ with $\langle x_q,V_l\rangle\neq0$. Since $\langle V_1,V_{l-1}\rangle=0$, we have $x_{q\prim}\in V_l$ and $x_{q\prime}\notin V_{l-1}$. Set $w_{q}=y_s$, $w_{q\prim}=y_t$, $w_s=z_{q}$, $w_t=z_{q\prime}$ and $w_j=z_j$ for $j\neq q,q\prim,s,t$. The relabelled basis still satisfies $V_i=\Span\{w_j\;|\;j\in \Xi_i\}$ since $z_{q\prim},y_t\in V_l$, and $z_{q},y_s\in V_1$. Also, $\langle w_{q},w_j\rangle=\langle w_{q\prim}, w_j\rangle=0$ for $j\neq q,q\prim$, $\langle w_q,w_{q\prim}\rangle\neq0$, and $w_j$ are eigenvectors of $\mathcal{Q}$.

Let $V\prim=\Span\{w_j\;|\;1\leq j\leq 2r,\, j\neq q,q\prim\}$. Then $V\prim$ is invariant under $\mathcal{Q}$, and the quasi-torus $\mathcal{Q}|_{V\prim}$ satisfies the conditions of the lemma with the flag
$${V}_i\prim=\Span\{w_j\;|\;j\in \Xi_i\setminus\{q,q\prim\}\}=V_i\cap {V\prim}.$$
Since $\dim {V\prim}<\dim V$, we can apply the induction hypothesis and find a basis $\{e_j\}_{j\neq q,q\prim}$, $1\leq j\leq 2r$, such that $\langle e_j,e_k\rangle=\sigma(j)\delta_{j,k\prim}$, where $e_j$ are eigenvectors of $\mathcal{Q}|_{V\prim}$ and ${V\prim}_i=\Span\{e_j\;|\;j\in \Xi_i\setminus\{q,q\prim\}\}$.

In order to have a complete basis for $V$ we set $e_q=\di\frac{\sigma(q)}{\langle w_q,w_{q\prim}\rangle}w_q$ and $e_{q\prim}=w_{q\prim}$. Then the basis $\{e_j\}_{j=1}^{2r}$ is a basis with the desired properties, and the induction step is proven.\end{proof}

In the following proofs we will deal with the differential forms $\wh$ and $\wk$. 

\begin{prp}\label{p2}
Let $\mathcal{Q}$ be a quasi-torus contained in $N_{\Hf}(\Tf_H)$. Then there exists $\psi \in\Spmn$ such that $\psi\mathcal{Q}\psi\inv\subset \Tf_H$. 
\end{prp}

\begin{proof}

By Lemma \ref{l22}, $N_{\Hf}(\Tf_H)\subset\M$, which we regard as a subgroup of $\GL(m)$. Recall that any element of $\Aut_0\Omn\cap\Hf$ can be written as  $\alpha S$ where $\alpha\in F^\times$ and $S\in\Spmn$. Let $$\mathcal{Q}\prim=\{S\in\Spmn\,|\,\mbox{there exists  }\alpha\in F^\times\mbox{ such that }\alpha S\in \mathcal{Q}\}.$$
 $\mathcal{Q}\prim$ is a quasi-torus since $\mathcal{Q}$ is a quasi-torus.  

Let $\{e_j\}_{j=1}^{m}$ be a basis as in Lemma \ref{l24} with respect to $\mathcal{Q}\prim$, and define $\psi:V\to V$ by $\psi(x_j)=e_j$ for $1\leq j\leq m$. Since $\langle\psi(x_j),\psi(x_k)\rangle=\langle e_j,e_k\rangle=\sigma(j)\delta_{j,k\prim}=\langle x_j,x_k\rangle$, we have $\psi\in \Sp(m)$. Since $V_i=\Span\{e_j\;|\;j\in\Xi_{i}\}=\Span\{x_j\;|\;j\in\Xi_{i}\}$, we have $\psi(V_i)=V_i$. Hence $\psi\in \Spmn$. Since $e_j$ are common eigenvectors of $\mathcal{Q}\prim$, we have $\psi\inv\mathcal{Q}\prim\psi\subset \Tf_H$.
Since every element of $\mathcal{Q}$ has the form $\alpha S$ with $\alpha\in F^\times$ and $S\in\mathcal{Q}\prim$, we have $\psi\inv\mathcal{Q}\psi\subset \mathcal{T}_H$. Replacing $\psi$ with $\psi\inv$, we get the result.\end{proof}

In order to get a similar result for the contact algebras, we use the Hamiltonian algebras contained in them. Let $m=2r+1$, $\underline{n}=(n_1,\dots,n_{2r+1})$ and $\underline{n}\prim=(n_1,\dots,n_{2r})$.

\begin{lmm}\label{l25}
Let $\psi\in\Aut_0 O(2r;\underline{n}\prim)$. If
$\psi(\wh)=\alpha\wh$ then there exists $\overline{\psi}\in\mathcal{K}(2r+1,\underline{n})$ such that $\overline{\psi}|_{O(2r;\underline{n}\prim)}=\psi$ and $\overline{\psi}(x_{2r+1})=\alpha x_{2r+1}$.
\end{lmm}

\begin{proof} 
Suppose $\psi\in\Aut_0 O(2r,\underline{n})$ given by $\psi(x_i)=\di\sum_{j=1}^{2r}\alpha_{i,j}x_j$ has the property $\psi(\wh)=\alpha\wh$. Since
$$\begin{array}{lll}\psi(\wh)&=&\psi\left(\di\sum_{i=1}^rdx_i\wedge dx_{i+r}  \right)=\di\sum_{i=1}^rd(\psi(x_i))\wedge d(\psi(x_{i+r}))\\
&=&\di\sum_{i=1}^rd\left(\di\sum_{j=1}^{2r}\alpha_{i,j}x_j\right)\wedge d\left(\di\sum_{k=1}^{2r}\alpha_{i+r,k}x_k\right)\\
&=&\di\sum_{1\leq j<k\leq 2r}\left(\sum_{i=1}^r\alpha_{i,j}\alpha_{i+r,k}-\alpha_{i,k}\alpha_{i+r,j}\right)dx_j\wedge dx_k
\end{array}$$
we have $\sum_{i=1}^r\alpha_{i,j}\alpha_{i+r,k}-\alpha_{i,k}\alpha_{i+r,j}=\delta_{k,j+r}\alpha$ for $1\leq j\leq r$.

Define $\overline{\psi}\in\Aut_c O(2r+1,\underline{n})$ by setting $\overline{\psi}(x_i)=\psi(x_i)$, $1\leq i\leq 2r$, and  $\overline{\psi}(x_m)=\alpha x_m$. Then
$$\begin{array}{lclcl}
\overline{\psi}(\wk)&=&\overline{\psi}\left( dx_m\right)&+&\overline{\psi}\left(\di\sum_{i=1}^{r}\left(x_idx_{i+r}-x_{i+r}dx_i\right)\right)\\
&=&d(\alpha x_m)&+&\di\sum_{i=1}^{r}\left(\left(\di\sum_{j=1}^{2r}\alpha_{i,j}x_j\right)d\left(\di\sum_{k=1}^{2r}\alpha_{i+r,k}x_k\right)\right.\\
&&&-&\left.\left(\di\sum_{k=1}^{2r}\alpha_{i+r,k}x_k\right)d\left(\di\sum_{j=1}^{2r}\alpha_{i,j}x_j\right)\right)\\ 
&=&\alpha dx_m&+&\di\di\sum_{1\leq j<k\leq 2r}\left(\di\sum_{i=1}^{r}\alpha_{i,j}\alpha_{i+r,k}-\alpha_{i,k}\alpha_{i+r,j}\right)x_jdx_k\\
&&&+&\di\di\sum_{1\leq j<k\leq 2r}\left(\di\sum_{i=1}^{r}\alpha_{i,k}\alpha_{i+r,j}-\alpha_{i,j}\alpha_{i+r,k}\right)x_kdx_j\\
&=&\alpha dx_m&+&\di\sum_{j=1}^{r}\alpha x_jdx_{j+r}+\di\sum_{j=1}^r(-\alpha) x_{j+r}dx_j\\
&=&\alpha\wk
\end{array}$$
Therefore, $\overline\psi\in \mathcal K(2r+1,\underline{n}).$\end{proof}
\begin{prp}\label{p10}
Let $\mathcal{Q}$ be a quasi-torus contained in $N_{\Kf}(\Tf_K)$. Then there exists $\psi\in\Kf$ such that $\psi\mathcal{Q}\psi\inv\subset \Tf_K$.
\end{prp}
\begin{proof}
Let $\mu\in \mathcal{Q}\subset N_{\Kf}(\Tf_K)$. By Lemma \ref{l23}, we have $\mu(x_i)=\alpha_ix_{j_i}$ for $1\leq i\leq 2r$ and
$\mu(x_m)=\alpha_mx_m+\di\sum_{l=1}^r\beta_lx_lx_{l\prim}$. Since $\mu\in\Kf$, we must have $\mu(\wk)\in\Omn^\times\wk$. On the other hand,
$$\begin{array}{lcl}\mu(\wk)&=&\mu\left(dx_m+\di\sum_{i=1}^r(x_idx_{i\prim}-x_{i\prim}dx_i)\right)\\
\\
&=&d\left(\alpha_mx_m+\di\sum_{l=1}^r\beta_lx_lx_{l\prim}\right)+\di\sum_{i=1}^r\alpha_i\alpha_{i\prim}(x_{j_i}dx_{{j_{i\prim}}}-x_{{j_{i\prim}}}dx_{j_i})\\
&=&\alpha_mdx_m+\di\sum_{l=1}^r\beta_l(x_ldx_{l\prim}+x_{l\prim}dx_l)+\di\sum_{i=1}^r\alpha_i\alpha_{i\prim}(x_{j_i}dx_{{j_{i\prim}}}-x_{{j_{i\prim}}}dx_{j_i})
\end{array}
$$
It follows that $\mu(\wk)=\alpha_m\wk$ since the only term with $dx_m$ is $\alpha_mdx_m$.

We want to show that $\mu|_{O(2r;\underline{n}\prim)}$ belongs to $\mathcal{H}(2r;\underline{n}\prim)$. Indeed,  $$\begin{array}{lll}d\wk&=&d\left(dx_m+\di\sum_{i=1}^r(x_idx_{i\prim}-x_{i\prim}dx_i)\right)=\di\sum_{i=1}^r(d(x_idx_{i\prim})-d(x_{i\prim}dx_i))\\
&=&\di\sum_{i=1}^r(dx_i\wedge dx_{i\prim}-dx_{i\prim}\wedge dx_i)=2\wh,\end{array}$$ 
and hence $$2\mu(\wh)=\mu(2\wh)=\mu(d\wk)=d\mu(\wk)=d(\alpha_m\wk)=2\alpha_m\wh,$$ so $\mu|_{O(2r,\underline{n}\prim)}\in \mathcal{H}(2r;\underline{n}\prim)$. 

We have shown that $N_{\Kf}(\Tf_K)|_{O(2r;\underline{n}\prim)}\subset \mathcal{H}(2r;\underline{n}\prim)$. Moreover, since the restriction of $\Tf_K$ to $O(2r;\unn\prim)$ is $\Tf_H$, where $\Tf_H$ is with respect to $\mathcal{H}(2r;\unn\prim)$, we have $N_{\Kf}(\Tf_K)|_{O(2r,\underline{n}\prim)}\subset N_{\mathcal{H}(2r;\underline{n}\prim)}(\Tf_H)$. By Proposition \ref{p2}, there exists $\psi\in\Sp(2r;\unn\prim)$ such that $\psi (\mathcal{Q}|_{O(2r,\underline{n}\prim)})\psi\inv\subset \Tf_H$.

By Lemma \ref{l25}, we can extend $\psi$ to an automorphism $\overline\psi$ in $\Kf$ such that $\overline\psi(x_i)=\psi(x_i)$ for $1\leq i\leq 2r$ and $\overline\psi(x_m)=x_m$ (since $\psi(\wh)=\wh$). Let $\mu\in\mathcal{Q}$ as before and set $\rho=\overline\psi\mu\overline\psi\inv$. Then $\rho(x_i)= \gamma_ix_i$ for $1\leq i\leq 2r$, $\gamma_i\in F^\times$, and $\rho(x_m)=\alpha_mx_m+y$ where $y=\psi\left(\sum_{l=1}^r\beta_lx_lx_{l\prim}\right)$. Furthermore,

$$\begin{array}{lll}\rho(\wh)&=&\overline\psi\mu\overline\psi\inv(\wh)=\overline\psi\mu(\wh)=\overline\psi(\alpha_m\wh)=\alpha_m\wh\quad\mbox{and also}\\
\\
\rho(\wh)&=&\rho\left(\di\sum_{l=1}^rdx_l\wedge dx_{l\prim}\right)=\di\sum_{l=1}^rd\rho(x_l)\wedge d\rho(x_{l\prim})=\di\sum_{l=1}^r\gamma_l\gamma_{l\prim}dx_l\wedge dx_{l\prim}.\end{array}$$
Hence we conclude that $\gamma_l\gamma_{l\prim}=\alpha_m$ for $1\leq l\leq r$ and
$$\begin{array}{lll}\rho(\wk)&=&\alpha_mdx_m+dy+\di\sum_{i=1}^r\gamma_i\gamma_{i\prim}(x_{j_i}dx_{{j_{i\prim}}}-x_{{j_{i\prim}}}dx_{j_i})\\
&=&\alpha_mx_m+dy+\alpha_m\di\sum_{l=1}^r(x_ldx_{l\prim}-x_{l\prim}dx_l)=\alpha_m\wk+dy.\end{array}$$
On the other hand, $\rho(\wk)=\overline{\psi}\mu\overline{\psi}\inv(\wk)=\overline{\psi}\mu(\wk)=\overline{\psi}(\alpha_m\wk)=\alpha_m\wk$. Since $dy=d\left(\psi\left(\sum_{l=1}^r\beta_lx_lx_{l\prim}\right)\right)=\psi\left(d\left(\sum_{l=1}^r\beta_lx_lx_{l\prim}\right)\right)$, we obtain: $$0=d\left(\sum_{l=1}^r\beta_lx_lx_{l\prim}\right)=\sum_{l=1}^r\beta_l(x_ldx_{l\prim}+x_{l\prim}dx_l),$$
 which implies $\beta_l=0$ for $1\leq l\leq r$. It follows that $\rho\in \Tf_K$. \end{proof}

\begin{prp}\label{c2} 
Let $L=\Wn$, $\Sn$, $\Hn$ or $\Kn$ and let $Q$ be a quasi-torus in $\Aut L$. Then there exists $\Psi\in\Aut L$ such that $\Psi Q\Psi\inv\subset T_X$ where $X=W$, $S$, $H$ or $K$, respectively.
\end{prp}

\begin{proof}
Let $Q$ be a quasi-torus in $\Aut L$. It follows that $\mathcal{Q}:=\Phi\inv(Q)$ is a quasi-torus in $\Xf$. By Proposition \ref{t10}, $\mathcal{Q}$ is contained in the normalizer of a maximal torus in $\Xf$. Since all maximal tori are conjugate, without loss of generality we may assume that $\mathcal{Q}\subset N_{\Xf}(\Tf_X)$.
Now Propositions \ref{p1}, 3, \ref{p10} show that there exists $\psi\in\Xf$ such that $\psi\mathcal{Q}\psi\inv\subset\Tf_X$. Hence
$$\Phi(\psi) Q(\Phi(\psi))\inv =\Phi(\psi)\Phi(\mathcal{Q})(\Phi(\psi))\inv=\Phi(\psi \mathcal{Q}\psi\inv)\subset\Phi(\mathcal{T}_X)=T_X.$$
\end{proof}
We can now prove Theorem \ref{tM}. 

\begin{proof}
Let $L=\Wn,$ $\Sn$, $\Hn$ or $\Kn$. Suppose $\Lgrad$ is a $G$-grading where $G$ is a group without elements of order $p$. Without loss of generality, we assume that the support of the grading generates $G$. Let $\eta:\widehat G\to\Aut L$ be the corresponding embedding and $Q:=\eta(\widehat G)$. Then $Q$ is a quasi-torus in $\Aut L$. By Proposition \ref{c2}, we can conjugate $Q$ by an automorphism $\Psi$ of $L$ so that $\Psi Q\Psi\inv\subset T_X$. It follows from Lemma \ref{le} that the grading $L\grad L\prim_g$, where $L\prim_g=\Psi(L_g)$, is a standard grading.
\end{proof}

\end{document}